\documentclass[twoside,11pt]{article}
\usepackage{amsmath, amssymb, mathrsfs, amsfonts, eucal, epsfig}

 \newtheorem{theorem}{Theorem}[section]
 \newtheorem{lemma}[theorem]{Lemma}
 \newtheorem{proposition}[theorem]{Proposition}
 \newtheorem{corollary}[theorem]{Corollary}
 \newtheorem{definition}[theorem]{Definition}
 \newenvironment{proof}{\begin{trivlist} \item[]{\em Proof.}}{\end{trivlist}}
 \newcount\refno
\def\CC{{\mathbb C}}
\def\DD{{\mathbb D}}

 \def\RR{{\mathbb R}}
 \def\NN{{\mathbb N}}

 \title{\bf A Bloch type space associated with $\lambda$-analytic functions
 \thanks{
 $^\dag$Corresponding author: Yeli Niu.
 \newline
 \indent\,\, E-mail: hhwei@cslg.edu.cn (H.-H. Wei); khqian2023@163.com (K.-H. Qian); lizk@shnu.edu.cn (Zh.-K. Li); niuyl@shnu.edu.cn (Y.-L. Niu).}
}

\author{Haihua Wei$^{1}$, Kanghui Qian$^{2}$, Zhongkai Li$^{2}$ and Yeli Niu$^{2,\dag}$\\
{\small $^{1}$School of Mathematics and Statistics, Suzhou University of Technology}\\
{\small Changshu 215500, Jiangsu, China}\\
{\small $^{2}$Department of Mathematics, Shanghai Normal University}\\
{\small Shanghai 200234, China}
}

\date{}

\begin{document}
 \maketitle \setcounter{page}{1} \pagestyle{myheadings}
 \markboth{Wei, Li and Li}{Bloch space associated with generalized analytic functions}

 \begin{abstract}
 \noindent

For $\lambda\ge0$, the so-called $\lambda$-analytic functions are defined in terms of the (complex) Dunkl operators $D_{z}$ and $D_{\bar{z}}$.
In the paper we introduce a Bloch type space on the disk $\DD$ associated with $\lambda$-analytic functions, called the $\lambda$-Bloch space and denoted by ${\mathfrak{B}}_{\lambda}(\DD)$. Various properties of the $\lambda$-Bloch space ${\mathfrak{B}}_{\lambda}(\DD)$ are proved. We give a characterization of functions in ${\mathfrak{B}}_{\lambda}(\DD)$ by means of the higher-order operators $(D_z\circ z)^n$ for $n\ge2$.
A general integral operator is proved to be bounded from $L^{\infty}(\DD)$ onto ${\mathfrak{B}}_{\lambda}(\DD)$, and as an application, the dual relation of ${\mathfrak{B}}_{\lambda}(\DD)$ and the $\lambda$-Bergman space ($p=1$) is verified.

 \vskip .2in
 \noindent
 {\bf 2020 MS Classification:} 30H20, 30H10 (Primary), 30G30, 42A45 (Secondary)
 \end{abstract}

 %%%%%%%%%%%%%%%%%%%%%%%%%%%%%%%%%%%%%%%%%%%%%%%%%%%%%%%%%%%%%%%%%%%%%%%%%%%%%%%%%%%%%%%%%%
\setcounter{page}{1}
%%%%%%%%%%%%%%%%%%%%%%%%%%%%%%%%%%%%%%%%%%%%%%%%%%%%%%%%%%%%%%%%%%%%%%%%%%%%%%%%%%%%%%%%%%

\section{Introduction}

In several works \cite{LL1,LW1,LW2,QW1}, the theories of the Hardy space and the Bergman space associated with the $\lambda$-analytic functions on the unit disk $\DD$ were developed, and in \cite{LL2,QWL,WLL}, their analogs on the upper half-plane were studied.
In this paper, we consider a Bloch type space associated with the $\lambda$-analytic functions on $\DD$.

The (complex) Dunkl operators $D_{z}$ and $D_{\bar{z}}$ in the complex plane $\CC$ are the substitutes of $\partial_z$ and $\partial_{\bar{z}}$, but involving a reflection term about the real axis respectively; and concretely, for $\lambda\ge0$ they are given by
\begin{align*}
    D_{z}f(z)=\partial_z f+
    \lambda \frac{f(z)-f(\bar{z})}{z-\bar{z}},\qquad
    D_{\bar{z}}f(z)=\partial_{\bar{z}}f-
      \lambda\frac{f(z)-f(\bar{z})}{z-\bar{z}}.
\end{align*}
For a domain $\Omega$ of $\CC$ that is symmetric about the real axis, a $C^2$ function $f$ defined on $\Omega$ is said to be
$\lambda$-analytic if $D_{\bar{z}}f\equiv0$.
The typical examples of $\lambda$-analytic functions are $z+\lambda(z+\bar{z})$ on $\CC$ and $1/(z|z|^{2\lambda})$ on $\CC\setminus\{0\}$.
Since for $\lambda\neq0$, $\lambda$-analytic functions are no longer differentiable about the complex variable $z$,
we always presuppose that they are in the $C^2$ class with respect to the real variables $x$ and $y$.

It was proved in \cite{LL1} that $f$ is $\lambda$-analytic in $\DD$ if and only if $f$ has the series representation
\begin{eqnarray}\label{anal-series-1-1}
f(z)=\sum_{n=0}^{\infty}c_{n}\phi_{n}^{\lambda}(z), \qquad |z|<1,
\end{eqnarray}
where
\begin{eqnarray*}
\phi_{n}^{\lambda}(z)=\epsilon_n\sum_{j=0}^{n}\frac{(\lambda)_{j}(\lambda+1)_{n-j}}
{j!(n-j)!}\bar{z}^{j}z^{n-j}, \qquad n\geq 0.
\end{eqnarray*}
with $\epsilon_n=\sqrt{n!/(2\lambda+1)_{n}}$. It is remarked that for $n\in\NN_0$ (the set of nonnegative integers), $\phi_n^{0}(z)=z^n$. In what follows we always assume that $\lambda>0$.

The measure on the unit disk $\DD$ associated with the operators $D_{z}$ and $D_{\bar{z}}$ is
\begin{eqnarray*}
d\sigma_{\lambda}(z)=c_{\lambda}|y|^{2\lambda}dxdy, \qquad z=x+iy,
\end{eqnarray*}
where $c_{\lambda}=\Gamma(\lambda+2)/\Gamma(\lambda+1/2)\Gamma(1/2)$ so that $\int_{\DD}d\sigma_{\lambda}(z)=1$.
Let $L_{\lambda}^p(\DD)$ denote the collection of measurable functions
$f$ on $\DD$ satisfying $\|f\|_{L_{\lambda}^p(\DD)}<\infty$, where $\|f\|_{L_{\lambda}^{p}(\DD)}=\left(\int_{\DD}|f(z)|^{p}d\sigma_{\lambda}(z)\right)^{1/p}$ for $0<p<\infty$, and $\|f\|_{L_{\lambda}^{\infty}(\DD)}=\|f\|_{L^{\infty}(\DD)}$ is given in the usual way.
The associated Bergman space $A^{p}_{\lambda}(\DD)$, named the $\lambda$-Bergman
space, consists of those elements in $L_{\lambda}^{p}(\DD)$ that are
$\lambda$-analytic in $\DD$, and the norm of $f\in A^{p}_{\lambda}(\DD)$ is written as $\|f\|_{A_{\lambda}^{p}}$ instead of $\|f\|_{L_{\lambda}^{p}(\DD)}$.

It follows from \cite[Theorem 6.8]{LW1} that, for $1\le p<\infty$ and for a function $f$ that is $\lambda$-analytic in $\DD$,
$f\in A^{p}_{\lambda}(\DD)$ if and only if $(1-|z|^{2})D_{z}\left(zf(z)\right)\in L^{p}_{\lambda}(\DD)$; and moreover
$$
\|f\|_{A_{\lambda}^{p}}\asymp\|(1-|z|^2)D_{z}\left(zf(z)\right)\|_{L_{\lambda}^{p}(\DD)}.
$$
For $p=\infty$, \cite[Lemma 6.6]{LW1} asserts that the mapping $f(z)\mapsto(1-|z|^{2})D_{z}\left(zf(z)\right)$ is
bounded from $A^{\infty}_{\lambda}(\DD)$ into $L^{\infty}_{\lambda}(\DD)$, but there exists an unbounded $\lambda$-analytic function on $\DD$
satisfying the condition $(1-|z|^{2})D_{z}\left(zf(z)\right)\in L^{\infty}_{\lambda}(\DD)$ (see (\ref{example-1}) later).
Based on these observations we now introduce a Bloch type space associated with the $\lambda$-analytic functions on $\DD$ as follows.

%From Proposition \ref{derivative-Bergman-a}, we have
%\begin{align}\label{boundedness-1}
%\sup_{z\in\DD}(1-|z|^2)|D_{z}\left(zf(z)\right)|\lesssim\|f\|_{A_{\lambda}^{\infty}}, \qquad f\in A^{\infty}_{\lambda}(\DD).
%\end{align}

\begin{definition}\label{Bloch-definition-a}
The $\lambda$-Bloch space ${\mathfrak{B}}_{\lambda}(\DD)$, or simply ${\mathfrak{B}}_{\lambda}$, consists of the $\lambda$-analytic functions on $\DD$ satisfying the condition
\begin{align}\label{Bloch-1}
\|f\|_{{\mathfrak{B}}_{\lambda}}:=\sup_{z\in\DD}(1-|z|^2)|D_{z}\left(zf(z)\right)|<\infty.
\end{align}
\end{definition}

The purpose of the paper is to study the $\lambda$-Bloch space ${\mathfrak{B}}_{\lambda}(\DD)$.

Note that for $\lambda=0$, the condition (\ref{Bloch-1}) is equivalent to
\begin{align}\label{Bloch-2}
|f(0)|+\sup_{z\in\DD}(1-|z|^2)|f'(z)|<\infty,
\end{align}
that is the defining form of the Bloch space ${\mathfrak{B}}(\DD)$ of the usual analytic functions.
The modern theory of the Bloch space originated from \cite{ACP} and \cite{Po1}, and its further development can be found
in \cite{An1,Ci1, DS1,HKZ1,Zhu1} and the references therein.
It plays a role in the theory of the Bergman space as the same as that ${\text {BMO}}$ plays in the theory of the Hardy space;
see the monographs \cite{DS1,HKZ1,Zhu1,Zhu2}.
However, the Bloch space has a longer history than the Bergman space, originating in a geometric form in the paper \cite{Blo1} of A. Bloch.
The condition for Bloch functions like (\ref{Bloch-2}) was motivated by Bloch's work and confirmed in \cite{Po1} and \cite{SeW1}.

It seems difficult to find the geometric correlation of the $\lambda$-Bloch space ${\mathfrak{B}}_{\lambda}(\DD)$ for $\lambda>0$,
and nevertheless, in the functional analytic aspect, this space is likely to drive many interesting topics and to play a significant role.
But it should be pointed out that, conformal mappings are no longer effective, since
product and composition of $\lambda$-analytic functions are nevermore $\lambda$-analytic in general.
Thus often times, a completely different approach must be employed.

The paper is organized as follows. Section 2 serves to review some basic knowledge about $\lambda$-analytic functions on the disk $\DD$, and
Section 3 is devoted to several fundamental properties of the $\lambda$-Bloch space ${\mathfrak{B}}_{\lambda}(\DD)$. In Section 4,
we give a characterization of functions in ${\mathfrak{B}}_{\lambda}(\DD)$ by means of the higher-order operators $(D_z\circ z)^n$ for $n\ge2$.
In the final section, a general integral operator is proved to be bounded from $L^{\infty}(\DD)$ onto ${\mathfrak{B}}_{\lambda}(\DD)$, and as an application, the dual relation of ${\mathfrak{B}}_{\lambda}(\DD)$ and the $\lambda$-Bergman space $A^1_{\lambda}(\DD)$ is verified.

The topic on the $\lambda$-analytic functions is motivated by C. Dunkl's work \cite{Dun3}, where he built up a framework associated with the dihedral group $G=D_k$ on the disk $\DD$. The researches in \cite{LL1,LW1,LW2,QW1} focus on the special case with $G=D_1$ having
the reflection $z\mapsto\overline{z}$ only, to find possibilities to develop a deep theory of associated function spaces. We note that C. Dunkl has a general theory named after him associated with reflection-invariance on the Euclidean spaces, see \cite{Dun1}, \cite{Dun2} and \cite{Dun4} for example.

Throughout the paper, the notation ${\mathcal{X}}\lesssim {\mathcal{Y}}$ or ${\mathcal{Y}}\gtrsim {\mathcal{X}}$ means that ${\mathcal{X}}\le c{\mathcal{Y}}$ for some positive constant $c$ independent of variables, functions, etc., and ${\mathcal{X}}\asymp {\mathcal{Y}}$ means that both ${\mathcal{X}}\lesssim {\mathcal{Y}}$ and ${\mathcal{Y}}\lesssim {\mathcal{X}}$ hold.

\section{Some facts on the $\lambda$-analytic functions}

For convenience of readers, we recall the basic theory of $\lambda$-analytic functions on the disk $\DD$, together with the associated harmonic functions.

For $0<p<\infty$, we denote by $L_{\lambda}^{p}(\partial\DD)$ the space of measurable functions $f$ on the circle $\partial\DD\simeq[-\pi,\pi]$ satisfying
$\|f\|_{L_{\lambda}^{p}(\partial\DD)}<\infty$, where $\|f\|_{L_{\lambda}^{p}(\partial\DD)}^p=\int_{-\pi}^{\pi}|f(e^{i\theta})|^p
dm_{\lambda}(\theta)$, and the measure $dm_{\lambda}$ on $\partial\DD$ is given by
\begin{eqnarray*}
dm_{\lambda}(\theta)=\tilde{c}_{\lambda}|\sin\theta|^{2\lambda}d\theta, \ \
\ \ \ \ \tilde{c}_{\lambda}=c_{\lambda}/(2\lambda+2).
\end{eqnarray*}
%Note that $L^{\infty}_{\lambda}(\partial\DD)$ is identical with the usual $L^{\infty}(\partial\DD)$.

It follows from \cite{Dun3,LL1} that the system
$$
\{\phi_{n}^{\lambda}(e^{i\theta})\}_{n=0}^{\infty}\cup\{e^{-i\theta}\phi_{n-1}^{\lambda}(e^{-i\theta})\}_{n=1}^{\infty}
$$
is an orthonormal basis of the Hilbert space $L_{\lambda}^2(\partial\DD)$. We note that $\phi_{0}^{\lambda}(z)\equiv1$,
and for $n\ge1$ and $z=re^{i\theta}$, from \cite[(1) and (3)]{LL1} we have
\begin{align}
&\phi_{n}^{\lambda}(z)=\epsilon_nr^{n}\left[\frac{n+2\lambda}{2\lambda}P_{n}^{\lambda}
(\cos\theta)+i\sin\theta P_{n-1}^{\lambda+1}(\cos\theta)\right],\label{basis-1}\\
&\bar{z}\overline{\phi_{n-1}^{\lambda}(z)}
=\epsilon_{n-1}r^{n}\left[\frac{n}{2\lambda}P_{n}^{\lambda}(\cos\theta)-
 i\sin\theta P_{n-1}^{\lambda+1}(\cos\theta)\right],\nonumber
\end{align}
where $P_{n}^{\lambda}(t)$ is the Gegenbauer
polynomial of degree $n$($\in\NN_0$) with parameter $\lambda$ (cf. \cite{Sz}).

In what follows, we write $\phi_n(z)=\phi_n^{\lambda}(z)$ for simplicity. According to \cite[(29)]{LL1}, one has
\begin{eqnarray}\label{phi-bound-1}
|\phi_n(z)|\le\epsilon_n^{-1}|z|^n\asymp (n+1)^{\lambda}|z|^n.
\end{eqnarray}
%\begin{eqnarray}\label{phi-bound-1}
%|\phi_n(z)|\le\epsilon_n^{-1}|z|^n\asymp (n+1)^{\lambda}|z|^n/\sqrt{\Gamma(2\lambda+1)}.
%\end{eqnarray}

The Laplacian associated with $D_{z}$ and $D_{\bar{z}}$, called the $\lambda$-Laplacian, is defined by
$\Delta_{\lambda}=4D_{z}D_{\bar{z}}=4D_{\bar{z}}D_{z}$,
which can be written explicitly as
\begin{eqnarray*}
\Delta_{\lambda} f=\frac{\partial^2 f}{ \partial x^2}+\frac
{\partial^2 f}{\partial y^2} +\frac{2\lambda}{y}\frac{\partial
f}{\partial y}-\frac{\lambda}{y^2} [f(z)-f(\bar{z})],\qquad z=x+iy.
\end{eqnarray*}
A $C^2$ function $f$ defined on $\DD$ is said to be $\lambda-$harmonic, if $\Delta_{\lambda}f=0$.

\begin{proposition} \label{anal-basis} {\rm(\cite[Proposition 2.2]{LL1})}
The functions $\phi_{n}(z)$ ($n\in\NN_0$) are $\lambda$-analytic and $\bar{z}\overline{\phi_{n-1}}(z)$ ($n\in\NN$) are $\lambda$-harmonic. Moreover, for $n\in\NN$,
$$
D_{z}\phi_{n}(z)=\sqrt{n(n+2\lambda)}\phi_{n-1}(z), \qquad D_{z}(\bar{z}\overline{\phi_{n-1}}(z))=-\lambda\phi_{n-1}(z),
$$
and
\begin{align}\label{Tzphi-2}
D_{z}(z\phi_{n-1}(z))=(n+\lambda)\phi_{n-1}(z).
\end{align}
\end{proposition}

A finite linear combination of elements in the system $\{\phi_{n}(z)\}_{n=0}^{\infty}$
is called a $\lambda$-analytic polynomial, and respectively, a finite linear combination of elements in the system
\begin{align}\label{harmonic-basis-1}
\{\phi_{n}(z)\}_{n=0}^{\infty}\cup\{\overline{z\phi_{n-1}(z)}\}_{n=1}^{\infty}
\end{align}
is called a $\lambda$-harmonic polynomial.
From \cite{Dun3}, the $\lambda$-Cauchy kernel $C(z,w)$ and the $\lambda$-Poisson kernel $P(z,w)$, which reproduce, associated with the measure $dm_{\lambda}$ on the circle $\partial\DD$,
all $\lambda$-analytic polynomials
and $\lambda$-harmonic polynomials respectively, are given by
\begin{align}
C(z,w)&=\sum_{n=0}^{\infty}\phi_{n}(z)\overline{\phi_{n}(w)},\label{Cauchy-kernel-2-1}\\
P(z,w)&=C(z,w)+\bar{z}w C(w,z).\nonumber
\end{align}
Note that the series in (\ref{Cauchy-kernel-2-1}) is convergent absolutely for $zw\in\DD$ and uniformly for $zw$ in a compact subset of $\DD$,
and by \cite[Theorems 1.3 and 2.1]{Dun3}, for $zw\in\DD$ we have
\begin{align*}
C(z,w)=\frac{1}{1-z\bar{w}}P_{0}(z,w), \qquad
P(z,w)=\frac{1-|z|^{2}|w|^{2}}{|1-z\bar{w}|^{2}}P_{0}(z,w),\nonumber
\end{align*}
where
\begin{align*}
P_{0}(z,w)&=\frac{1}{|1-zw|^{2\lambda}}{}_2\!F_{1}\Big({\lambda,\lambda
\atop
  2\lambda+1};\frac{4({\rm Im} z)({\rm Im}
  w)}{|1-zw|^{2}}\Big)\nonumber\\
&=\frac{1}{|1-z\bar{w}|^{2\lambda}}{}_2\!F_{1}\Big({\lambda,\lambda+1
\atop
  2\lambda+1};-\frac{4({\rm Im} z)({\rm Im}
  w)}{|1-z\bar{w}|^{2}}\Big),
\end{align*}
and ${}_2\!F_{1}[a,b;c;t]$ is the Gauss hypergeometric function.

A $\lambda$-harmonic function in $\DD$ has a series representation
in terms of the system (\ref{harmonic-basis-1}) as given in the following proposition.

\begin{proposition} \label{har-series} {\rm(\cite[Theorem 3.1]{LL1})}
If $f$ is a $\lambda$-harmonic function in $\DD$, then there are two
sequences $\{c_n\}$ and $\{\tilde{c}_n\}$ of complex numbers, such
that
\begin{eqnarray*}
f(z)=\sum_{n=0}^{\infty}
c_n \phi_{n}(z)+
  \sum_{n=1}^{\infty} \tilde{c}_n \bar{z}\overline{\phi_{n-1}(z)}
\end{eqnarray*}
for $z\in\DD$. Moreover, for each real $\gamma$, the series
$\sum_{n\ge1}n^{\gamma}(|c_n|+|\tilde{c}_n|)r^n$ converges uniformly
for $r$ in every closed subset of $[0,1)$.
\end{proposition}

As stated in the first section, a $\lambda$-analytic
function $f$ on $\DD$ has a series representation as in (\ref{anal-series-1-1}); and moreover, such an $f$ could be characterized by a Cauchy-Riemann type system.

\begin{proposition} \label{anal-thm} {\rm(\cite[Theorem 3.7]{LL1})}
For a $C^2$ function $f=u+iv$ defined on $\DD$,
the following statements are equivalent:

{\rm (i)}  $f$ is $\lambda$-analytic;

{\rm (ii)} $u$ and $v$ satisfy the generalized Cauchy-Riemann equations
\begin{align*}
\partial_{x}u=D_{y}v \quad \hbox{and} \quad  D_{y}u=-\partial_{x}v,
\end{align*}
where
\begin{align*}
D_y u(x,y)=\partial_y u(x,y)+\frac{\lambda}{y}\left[u(x,y)-u(x,-y)\right];
\end{align*}

{\rm (iii)}  $f$ has the series representation in (\ref{anal-series-1-1}),
where for each real $\gamma$, the series
$\sum_{n\ge1}n^{\gamma}|c_n|r^n$ converges uniformly for $r$ in
every closed subset of $[0,1)$.
\end{proposition}

%The next proposition states a relation of $\lambda$-analytic functions and the usual analytic functions.
%
%\begin{proposition} \label{anal-anal} {\rm(\cite[Theorem 3.9]{LL1})}
%For a function $f$ defined in $\DD$, $f$ is $\lambda-$analytic in
%$\DD$ if and only if there exists an analytic function $f_0$ in
%$\DD$ such that
%$$
%f(z)=A_{\lambda}f_0(z),
%$$
%where
%\begin{eqnarray*} \label{A-operator-2}
%A_{\lambda}f_0(z)=\frac{\tilde{c}_{\lambda-1/2}}{4^{\lambda}\,i}\int_{|\xi|=1}f\left(x+iy\,{\rm Re}\,\xi\right)
%\left(1+\xi^{-1}\right)^2 \left|1-\xi^2\right|^{2\lambda-1}d\xi
%\end{eqnarray*}
%with $z=x+iy$.
%The inversion of $A_{\lambda}$ is given by
%\begin{eqnarray*}
%A_{\lambda}^{-1}f(w)=\frac{\tilde{c}_{\lambda}}{2^{2\lambda}i}\oint_{|z|=r}
%\frac{f(z)}{(\bar{z}-w)^{\lambda}(z-w)^{\lambda+1}}|z-\bar{z}|^{2\lambda}dz
%\end{eqnarray*}
%for $|w|<r<1$.
%\end{proposition}

As usual, the $p$-means of a function $f$ defined on $\DD$, for $0<p<\infty$, are given by
\begin{eqnarray*}
M_p(f;r)=\left\{\int_{-\pi}^{\pi}|f(re^{i\theta})|^p\,dm_{\lambda}(\theta)\right\}^{1/p},\qquad 0\le r<1;
\end{eqnarray*}
and $M_{\infty}(f;r)=\sup_{\theta}|f(re^{i\theta})|$.
The $\lambda$-Hardy space $H_{\lambda}^p(\DD)$ is the collection of $\lambda$-analytic functions on $\DD$ satisfying
$$
\|f\|_{H_{\lambda}^p}:=\sup_{0\le r <1}M_p(f;r)<\infty.
$$
Obviously $H_{\lambda}^{\infty}(\DD)$ is identical with $A_{\lambda}^{\infty}(\DD)$.

The fundamental theory of the $\lambda$-Hardy spaces $H_{\lambda}^p(\DD)$ for
$$
p\ge p_0:=\frac{2\lambda}{2\lambda+1}
$$
was studied in \cite{LL1}. The following theorem asserts the existence of boundary values of functions in $H_{\lambda}^p(\DD)$.

\begin{theorem} \label{Hardy-boundary-value-a} {\rm (\cite[Theorem 6.6]{LL1})}
Let $p\ge p_0$ and $f\in H_{\lambda}^p(\DD)$. Then for almost every $\theta\in[-\pi,\pi]$, $\lim f(r
e^{i\varphi})=f(e^{i\theta})$  exists as $r e^{i\varphi}$ approaches to the point
$e^{i\theta}$ nontangentially, and if $p_0<p<\infty$, then
\begin{align*}
\lim_{r\rightarrow1-}\int_{-\pi}^{\pi}|f(r
e^{i\theta})-f(e^{i\theta})|^{p} dm_{\lambda}(\theta)=0
\end{align*}
and $\|f\|_{H^{p}_{\lambda}}^p\asymp\int_{-\pi}^{\pi}|f(e^{i\theta})|^{p} dm_{\lambda}(\theta)$.
\end{theorem}

By \cite[Theorem 5.5]{LW1}, the $\lambda$-Hardy space $H^p_{\lambda}(\DD)$ for $p_0\le p\le\infty$ is complete, and by \cite[Theorem 5.2]{LW1}, the set of $\lambda$-analytic polynomials is dense in $H_{\lambda}^p(\DD)$ for $p_0<p<\infty$. In particular, the set $\{\phi_n(z): \,n\in\NN_0\}$ is an orthonormal basis of $H_{\lambda}^2(\DD)$. If $1<p<\infty$, \cite[Theorem 5.10]{LW1} asserted that the dual of $H^p_{\lambda}(\DD)$ is isomorphic to $H^{p'}_{\lambda}(\DD)$ with equivalent norms, where $1/p+1/p'=1$.

%Most materials of this section come from \cite{LL1}. Such a topic is motivated by C. Dunkl's work \cite{Dun3}, where he built up a framework associated with the dihedral group
%$G=D_k$ on the disk $\DD$. Our work here and also that in \cite{LL1} and \cite{LW1} focus on the special case with $G=D_1$ having
%the reflection $z\mapsto\overline{z}$ only, to find possibilities to develop a deep theory of associated function spaces. We note that C. Dunkl has a general theory named after him associated with reflection-invariance on the Euclidean spaces, see \cite{Dun1},\cite{Dun2} and \cite{Dun4} for example. For the Hardy spaces in the upper half plane $\RR_+^2$ associated to the Dunkl operators $D_{z}$ and $D_{\bar{z}}$, see \cite{LL2}.

The Bergman kernel on the $\lambda$-Bergman spaces $A^{2}_{\lambda}(\DD)$ is given by (cf. \cite[Section 3]{LW1})
\begin{align}\label{Bergman-k-1}
K_{\lambda}(z,w)=\sum_{n=0}^{\infty}
\frac{n+\lambda+1}{\lambda+1}\phi_{n}(z)\overline{\phi_{n}(w)}, \qquad |zw|<1,
\end{align}
and is called the $\lambda$-Bergman kernel. For $f\in L_{\lambda}^{1}(\DD)$, we define the $\lambda$-Bergman projection $P_{\lambda}$ by
\begin{align}\label{Bergman-projection-1}
(P_{\lambda}f)(z)=\int_{\DD}f(w)K_{\lambda}(z,w)d\sigma_{\lambda}(w),\qquad z\in\DD.
\end{align}
By \cite[Theorem 3.6]{LW1}, the operator $P_{\lambda}$ is bounded from $L_{\lambda}^{p}(\DD)$ onto $A^{p}_{\lambda}(\DD)$ for
$1<p<\infty$, and by \cite[Proposition 3.1]{LW1}, all $f\in A^{1}_{\lambda}(\DD)$ satisfy the reproducing formula $f=P_{\lambda}f$, i.\,e.,
\begin{eqnarray}\label{reproducing-Bergman-1}
f(z)=\int_{\DD}f(w)K_{\lambda}(z,w)d\sigma_{\lambda}(w),\qquad z\in\DD.
\end{eqnarray}
%Conversely, if $f\in L_{\lambda}^{1}(\DD)$ satisfies (\ref{reproducing-Bergman-1}), then $f$ is $\lambda$-analytic in $\DD$, and in particular $f\in A^{1}_{\lambda}(\DD)$.

For $f\in A_{\lambda}^p(\DD)$ with $p\ge p_0$, its point evaluation is given by (cf. \cite[(41)]{LW1})
\begin{align*}
|f(z)|\lesssim\frac{(1-|z|)^{-2/p}}{|1-z^2|^{2\lambda/p}}\|f\|_{A_{\lambda}^{p}},\qquad z\in\DD.
\end{align*}
By \cite[Theorem 5.6]{LW1}, the $\lambda$-Bergman spaces $A^{p}_{\lambda}(\DD)$ for $p_0\le p\le\infty$ is complete,
and by \cite[Theorem 5.3]{LW1}, the set of $\lambda$-analytic polynomials is dense in $A^{p}_{\lambda}(\DD)$ for $p_0<p<\infty$.
In particular, the set $\{a_n\phi_{n}^{\lambda}(z)\}_{n=0}^{\infty}$ forms an orthonormal basis of $A^{2}_{\lambda}(\DD)$, where $a_n=\sqrt{(n+\lambda+1)/(\lambda+1)}$ for $n\in\NN_0$. If $1<p<\infty$, \cite[Theorem 5.11]{LW1} showed that the dual of $A^p_{\lambda}(\DD)$ is isomorphic to $A^{p'}_{\lambda}(\DD)$ in the sense that, each $L\in A^p_{\lambda}(\DD)^*$ can be represented by
\begin{align*}
L(f)=\int_{\DD}f(z)\overline{g(z)}d\sigma_{\lambda}(z),\qquad f\in A^p_{\lambda}(\DD),
\end{align*}
with a unique function $g\in A^{p'}_{\lambda}(\DD)$ satisfying $C_p\|g\|_{A_{\lambda}^{p'}}\le\|L\|\le\|g\|_{A_{\lambda}^{p'}}$,
where the constant $C_p$ is independent of $g$.

%\begin{proposition}\label{derivative-Bergman-a}{\rm (\cite[Lemma 6.6]{LW1})}
%For $1\le p\le\infty$, the mapping $f(z)\mapsto(1-|z|^{2})D_{z}\left(zf(z)\right)$ is bounded from $A^{p}_{\lambda}(\DD)$ into $L^{p}_{\lambda}(\DD)$.
%\end{proposition}
%
%Furthermore we have
%
%\begin{theorem}\label{derivative-Bergman-c} {\rm (\cite[Theorem 6.8]{LW1})}
%Suppose that $1\le p<\infty$ and $f$ is $\lambda$-analytic in $\DD$. Then $f\in A^{p}_{\lambda}(\DD)$ if and only if $(1-|z|^{2})D_{z}\left(zf(z)\right)\in L^{p}_{\lambda}(\DD)$. Moreover
%$$
%\|f\|_{A_{\lambda}^{p}}\asymp\|(1-|z|^2)D_{z}\left(zf(z)\right)\|_{L_{\lambda}^{p}(\DD)}.
%$$
%\end{theorem}

For the Hardy spaces $H^p$ and the Bergman spaces on the upper half-plane $\RR_+^2$ associated to the $\lambda$-analytic functions, see \cite{LL2,QWL,WLL}, and for the Hardy space $H^1$ in the general Dunkl setting, see \cite{JL1}.

\section{Fundamental properties of the $\lambda$-Bloch space}

In this section we shall prove several properties of the $\lambda$-Bloch space ${\mathfrak{B}}_{\lambda}(\DD)$ defined in Definition \ref{Bloch-definition-a}.

\begin{lemma}\label{analytic-converge-a}{\rm (\cite[Lemma 5.4]{LW1})}
Let $\{f_n\}$ be a sequence of $\lambda$-analytic functions on $\DD$. If $\{f_n\}$ converges uniformly on each compact subset of $\DD$, then its limit function is also $\lambda$-analytic in $\DD$.
\end{lemma}

\begin{proposition}\label{Bloch-a}
The $\lambda$-Bloch space ${\mathfrak{B}}_{\lambda}(\DD)$ is a Banach space with the norm $\|\cdot\|_{{\mathfrak{B}}_{\lambda}}$ given in (\ref{Bloch-1}).
\end{proposition}

\begin{proof}
We note that $\|\cdot\|_{{\mathfrak{B}}_{\lambda}}$ is a norm. It suffices to verify that $\|f\|_{{\mathfrak{B}}_{\lambda}}=0$ implies $f\equiv0$.
Indeed, if $f(z)=\sum_{k=0}^{\infty}c_{k}\phi_{k}(z)$, it follows from (\ref{Tzphi-2}) that
$$
\sum_{k=0}^{\infty}(k+\lambda+1)c_k\phi_{k}^{\lambda}(z)=D_{z}\left(zf(z)\right)\equiv0,
$$
which certainly asserts that all $c_k=0$ for $k\in\NN_0$. Therefore $f\equiv0$.

Suppose $\{f_n\}_{n=1}^{\infty}$ is a Cauchy sequence in ${\mathfrak{B}}_{\lambda}(\DD)$ and $\DD_r=\{z\in\CC:\,|z|<r\}$ for $0<r<1$. Since
\begin{eqnarray*}
|D_{z}\left(zf_m(z)\right)-D_{z}\left(zf_n(z)\right)|\le(1-r^2)^{-1}\|f_m-f_n\|_{{\mathfrak{B}}_{\lambda}} \quad \hbox{for}\,\,z\in\overline{\DD}_r,
\end{eqnarray*}
it follows that
$\{D_{z}\left(zf_n(z)\right)\}_{n=1}^{\infty}$
converges to a function $g$ uniformly on each compact subset of $\DD$. By Lemma \ref{analytic-converge-a}, $g$ is $\lambda$-analytic in $\DD$. Assume $g(z)=\sum_{k=0}^{\infty}b_{k}\phi_{k}(z)$ and define
$$
f(z)=\sum_{k=0}^{\infty}\frac{b_{k}}{k+\lambda+1}\phi_{k}(z), \qquad z\in\DD.
$$
It then follows from (\ref{Tzphi-2}) that $D_{z}\left(zf(z)\right)=g(z)$. Thus, for $z\in\DD$ we have
\begin{align*}
&(1-|z|^2)|D_{z}\left(zf_n(z)\right)-D_{z}\left(zf(z)\right)|\\
&\qquad =\lim_{m\rightarrow\infty}(1-|z|^2)|D_{z}\left(zf_n(z)\right)-D_{z}\left(zf_m(z)\right)|\\
&\qquad \le\liminf_{m\rightarrow\infty}\|f_n-f_m\|_{{\mathfrak{B}}_{\lambda}},
\end{align*}
so that $\|f_n-f\|_{{\mathfrak{B}}_{\lambda}}\le\liminf\limits_{m\rightarrow\infty}\|f_n-f_m\|_{{\mathfrak{B}}_{\lambda}}$. Therefore $\lim\limits_{n\rightarrow\infty}\|f_n-f\|_{{\mathfrak{B}}_{\lambda}}=0$, and the completeness of the space ${\mathfrak{B}}_{\lambda}(\DD)$ is proved.
\end{proof}

\begin{proposition}\label{Bloch-b}
We have $A^{\infty}_{\lambda}(\DD)\subseteq{\mathfrak{B}}_{\lambda}$, and $\|f\|_{{\mathfrak{B}}_{\lambda}}\lesssim\|f\|_{A_{\lambda}^{\infty}}$ for $f\in A^{\infty}_{\lambda}(\DD)$.
\end{proposition}

Indeed, by \cite[Lemma 6.6]{LW1}, one has
\begin{align*}
\sup_{z\in\DD}(1-|z|^2)|D_{z}\left(zf(z)\right)|\lesssim\|f\|_{A_{\lambda}^{\infty}}, \qquad f\in A^{\infty}_{\lambda}(\DD).
\end{align*}
Thus the proposition is concluded.

For $\beta\in(-\infty,\infty)$, define the function $h_{\lambda,\beta}(z,w)$ by
\begin{align}\label{h-kernel-1}
h_{\lambda,\beta}(z,w)=\sum_{n=0}^{\infty} a_{\beta}(n)\phi_{n}(z)\overline{\phi_{n}(w)}, \qquad |zw|<1,
\end{align}
where $a_{\beta}(n)$ satisfies
\begin{align}\label{h-kernel-2}
a_{\beta}(n)=\sum_{j=0}^{M}a_{\beta,j}(n+1)^{\beta-j}+O\left((n+1)^{\beta-M-1}\right)
\end{align}
for $n\ge0$ and $M=\max\{[\beta+2\lambda+1],0\}$.

%The following lemma is a consequence of \cite[Corollary 7.3]{LW2}.

%\begin{lemma}\label{h-kernal-a}
%Let the function $h_{\lambda,\beta}(z,w)$ be defined by (\ref{h-kernel-1}) and (\ref{h-kernel-2}). Then for $\beta>0$,
%\begin{align*}
%|h_{\lambda,\beta}(z,w)|\lesssim \frac{(|1-z\overline{w}|+|1-zw|)^{-2\lambda}}{|1-z\overline{w}|}
%\left(\frac{1}{|1-z\overline{w}|^{\beta}}+\frac{1}{|1-zw|^{\beta}}\right),\quad z,w\in\DD.
%\end{align*}
%\end{lemma}

The following lemma is a consequence of \cite[Corollary 7.3]{LW2}, and will be often used subsequently.

\begin{lemma}\label{h-kernal-a}
Let the function $h_{\lambda,\beta}(z,w)$ be defined by (\ref{h-kernel-1}) and (\ref{h-kernel-2}). Then

{\rm (i)} for $\beta>0$,
\begin{align*}
|h_{\lambda,\beta}(z,w)|\lesssim \frac{(|1-z\overline{w}|+|1-zw|)^{-2\lambda}}{|1-z\overline{w}|}
\left(\frac{1}{|1-z\overline{w}|^{\beta}}+\frac{1}{|1-zw|^{\beta}}\right),\quad z,w\in\DD;
\end{align*}

%{\rm (ii)} for $\beta=0$,
%\begin{align*}
%|h_{\lambda,0}(z,w)|\lesssim \frac{(|1-z\overline{w}|+|1-zw|)^{-2\lambda}}{|1-z\overline{w}|}
%\ln\left(\frac{|1-z\overline{w}|}{|1-zw|}+2\right),\quad z,w\in\DD;
%\end{align*}

{\rm (ii)} for $-1<\beta<0$,
\begin{align*}
|h_{\lambda,\beta}(z,w)|\lesssim \frac{(|1-z\overline{w}|+|1-zw|)^{-2\lambda}}{|1-z\overline{w}|^{\beta+1}},\quad z,w\in\DD;
\end{align*}

{\rm (iii)} for $\beta=-1$,
\begin{align*}
|h_{\lambda,-1}(z,w)|\lesssim (|1-z\overline{w}|+|1-zw|)^{-2\lambda}\ln\left(\frac{|1-zw|}{|1-z\overline{w}|}+2\right),\quad z,w\in\DD;
\end{align*}

{\rm (iv)} for $-2\lambda-1<\beta<-1$,
\begin{align*}
|h_{\lambda,\beta}(z,w)|\lesssim (|1-z\overline{w}|+|1-zw|)^{-\beta-2\lambda-1},\quad z,w\in\DD.
\end{align*}

%{\rm (vi)} for $\beta=-2\lambda-1$,
%\begin{align*}
%|h_{\lambda,\beta}(z,w)|\lesssim \ln\left(\frac{1}{|1-z\overline{w}|+|1-zw|}+2\right),\quad z,w\in\DD;
%\end{align*}
%
%{\rm (vii)} for $\beta<-2\lambda-1$, $h_{\lambda,\beta}(z,w)$ is continuous for $z,w$ satisfying $|zw|\le1$.
\end{lemma}

\begin{lemma}\label{representation-a}{\rm (\cite[Proposition 6.9]{LW1})}
If $f$ is $\lambda$-analytic in $\DD$ and satisfies the condition $(1-|z|^{2})D_{z}\left(zf(z)\right)\in L^{1}_{\lambda}(\DD)$, then
\begin{align}\label{representation-1}
f(z)=\int_{\DD}D_{w}\left(wf(w)\right)\widetilde{K}_{\lambda}(z,w)(1-|w|^{2})\,d\sigma_{\lambda}(w),\qquad z\in\DD,
\end{align}
where
\begin{align}\label{kernel-1}
\widetilde{K}_{\lambda}(z,w)=\sum_{n=0}^{\infty}
\frac{n+\lambda+2}{\lambda+1}\phi_{n}(z)\overline{\phi_{n}(w)}.
\end{align}
\end{lemma}

The next proposition indicates a radial growth order of $f\in{\mathfrak{B}}_{\lambda}$ as $|z|\rightarrow1-$.

\begin{proposition}\label{Bloch-b}
If $f\in{\mathfrak{B}}_{\lambda}(\DD)$, then
$$
|f(z)|\lesssim\|f\|_{{\mathfrak{B}}_{\lambda}}\ln\frac{2}{1-|z|}, \qquad |z|<1.
$$
\end{proposition}

\begin{proof}
From (\ref{representation-1}) we have
\begin{align*}
|f(z)|\le\|f\|_{{\mathfrak{B}}_{\lambda}}\int_{\DD}|\widetilde{K}_{\lambda}(z,w)|\,d\sigma_{\lambda}(w),\qquad z\in\DD.
\end{align*}
For $z=re^{i\theta}$, $w=se^{i\varphi}\in\DD$, it is not difficult to verify the following inequalities
\begin{align}
&|1-z\overline{w}|\asymp1-rs+\left|\sin(\theta-\varphi)/2\right|,\label{inequality-1}\\
&|1-z\overline{w}|+|1-zw|\gtrsim 1-rs+|\sin\theta|+|\sin\varphi|,\label{inequality-2}
\end{align}
and according to Lemma \ref{h-kernal-a} (i) with $\beta=1$,
\begin{align*}
|\widetilde{K}_{\lambda}(z,w)|\lesssim \Phi_{r,\theta}(s,\varphi)+\Phi_{r,\theta}(s,-\varphi),
\end{align*}
where
\begin{align*}
 \Phi_{r,\theta}(s,\varphi)
 =\frac{\left(1-rs+\left|\sin(\theta-\varphi)/2\right|\right)^{-2}}{\left(1-rs+|\sin\theta|+|\sin\varphi|\right)^{2\lambda}}.
\end{align*}
Since the contribution of $\Phi_{r,\theta}(s,-\varphi)$ to the integral is the same as that of $\Phi_{r,\theta}(s,\varphi)$, one has
\begin{align*}
|f(z)| &\lesssim\|f\|_{{\mathfrak{B}}_{\lambda}}
\int_0^1\int_{-\pi}^{\pi}\Phi_{r,\theta}(s,\varphi)|\sin\varphi|^{2\lambda}d\varphi ds\\
&\le\|f\|_{{\mathfrak{B}}_{\lambda}}
\int_0^1\int_{-\pi}^{\pi}\frac{d\varphi ds}{\left(1-rs+\left|\sin(\theta-\varphi)/2\right|\right)^{2}}.
\end{align*}
Direct calculations show that the last double integral is dominated by a multiple of $\ln\frac{2}{1-r}$ for $|z|=r<1$. This finishes the proof of Proposition \ref{Bloch-b}.
\end{proof}

We now give an example which shows that  $A^{\infty}_{\lambda}(\DD)$ is a proper subset of ${\mathfrak{B}}_{\lambda}(\DD)$. We shall need a lemma.

For $\alpha>-1$, define the function $F_{\alpha}(z)$ by
\begin{align}\label{F-function-1}
F_{\alpha}(z)=\sum_{n=1}^{\infty} c_{\alpha}(n)r^nP_{n}^{\lambda}(\cos\theta), \qquad z=re^{i\theta}\in\DD,
\end{align}
where $c_{\alpha}(n)$ satisfies
\begin{align}\label{F-function-2}
c_{\alpha}(n)=\sum_{j=0}^{M}c_{\alpha,j}n^{-\alpha-j}+O\left((n+1)^{-\alpha-M-1}\right)
\end{align}
for $n\ge1$ and $M=\max\{[2\lambda-\alpha],0\}$.

\begin{lemma}\label{F-function-a}
For $\alpha>-1$, let the function $F_{\alpha}(z)$ be defined by (\ref{F-function-1}) and (\ref{F-function-2}), and $F(e^{i\theta})=\lim_{r\rightarrow1^-}F(re^{i\theta})$ whenever the limit exists for given $\theta$. Then

{\rm (i)} $F_{\alpha}(z)$ is continuous on ${\overline{\DD}}\setminus\{1\}$, and $F(e^{i\theta})\in L_{\lambda}^{p}(\partial\DD)$;

{\rm (ii)} if $\alpha>2\lambda$, then $F\in C({\overline{\DD}})$;

{\rm (iii)} if $\alpha=2\lambda$, then $F(e^{i\theta})\asymp\ln|\theta|^{-1}$ as $\theta\rightarrow0$;

{\rm (iv)} if $-1<\alpha<2\lambda$, then $F(e^{i\theta})\asymp|\theta|^{\alpha-2\lambda}$ as $\theta\rightarrow0$.
\end{lemma}

The assertions in the above lemma are special cases of those in \cite[Theorems 1 and 3]{AW}

\begin{proposition}
The function defined by
\begin{align}\label{example-1}
f_0(z)=\sum_{n=1}^{\infty}\frac{\phi_{n}^{\lambda}(z)}{n^{\lambda+1}}, \qquad |z|<1,
\end{align}
is in ${\mathfrak{B}}_{\lambda}(\DD)$, but not bounded on $\DD$.
\end{proposition}

\begin{proof}
It follows from (\ref{Tzphi-2}) that
\begin{align*}
D_{z}\left(zf_0(z)\right)=\sum_{n=1}^{\infty}\frac{n+\lambda+1}{n^{\lambda+1}}\phi_{n}^{\lambda}(z), \qquad |z|<1.
\end{align*}
Since (cf. \cite[(4.7.3)]{Sz}) $P_{n}^{\lambda}(1)=(2\lambda)_n/n!$, from (\ref{basis-1}) one has
\begin{align*}
\phi_{n}^{\lambda}(1)=\epsilon_n\times\frac{n+2\lambda}{2\lambda}P_{n}^{\lambda}(1)
=\epsilon_n\times\frac{(2\lambda+1)_n}{n!}
=\epsilon_n^{-1},
\end{align*}
so that
\begin{align*}
D_{z}\left(zf_0(z)\right)=\sum_{n=1}^{\infty}\frac{n+\lambda+1}{n^{\lambda+1}}\,
\epsilon_n\phi_{n}^{\lambda}(z)\overline{\phi_{n}^{\lambda}(1)}, \qquad |z|<1.
\end{align*}
It is easy to see that, for $n\ge1$, the function $n\mapsto\frac{n+\lambda+1}{n^{\lambda+1}}\,\epsilon_n$ has the expansion (\ref{h-kernel-2}) with $\beta=-2\lambda$, and hence $D_{z}\left(zf_0(z)\right)$ is identical with some $h_{\lambda,-2\lambda}(z,1)$. Applying Lemma \ref{h-kernal-a}, part (ii) for $0<\lambda<1/2$, part (iii) for $\lambda=1/2$, and part (iv) for $\lambda>1/2$ respectively, gives
\begin{align*}
|D_{z}\left(zf_0(z)\right)|\lesssim|1-z|^{-1}, \qquad |z|<1,
\end{align*}
so that $(1-|z|^2)|D_{z}\left(zf_0(z)\right)|\lesssim1$ for $|z|<1$. Therefore $f_0\in{\mathfrak{B}}_{\lambda}(\DD)$.

To show that $f_0$ is unbounded on $\DD$, we use (\ref{basis-1}) to get
\begin{align*}
{\mathrm Re}\, f_0(z)=\sum_{n=1}^{\infty}\frac{\epsilon_n}{n^{\lambda+1}}\frac{n+2\lambda}{2\lambda}r^{n}P_{n}^{\lambda}
(\cos\theta), \qquad z=re^{i\theta}\in\DD.
\end{align*}
Obviously for $n\ge1$, the function $n\mapsto\frac{\epsilon_n}{n^{\lambda+1}}\frac{n+2\lambda}{2\lambda}$ has the expansion (\ref{F-function-2}) with $\alpha=2\lambda$, and hence ${\mathrm Re}\, f_0(z)$ is identical with some $F_{2\lambda}(z)$. According to Lemma \ref{F-function-a} (i) and (iii), the function $f_0$ is unbounded on $\DD$.
\end{proof}

\section{Characterization of the $\lambda$-Bloch space by higher operators}

Although, as in (\ref{Bloch-1}), the $\lambda$-Bloch norm $\|\cdot\|_{{\mathfrak{B}}_{\lambda}}$ is defined by the first-order operator $D_z\circ z$,
it can also be characterized by the higher-order operators $(D_z\circ z)^n$ for $n\ge2$, as given in Theorem \ref{derivative-Bloch-a} below.

We shall need the following extension of Lemma \ref{representation-a}.

\begin{lemma}\label{representation-b}
If $f$ is $\lambda$-analytic in $\DD$ and satisfies the condition $(1-|z|^{2})^{\alpha} D_{z}\left(zf(z)\right)\in L^{1}_{\lambda}(\DD)$ for some $\alpha>-1$, then
\begin{align}\label{representation-2}
f(z)=\int_{\DD}D_{w}\left(wf(w)\right)\widetilde{K}_{\lambda,\alpha}(z,w)(1-|w|^{2})^{\alpha}\, d\sigma_{\lambda}(w),\qquad z\in\DD,
\end{align}
where
\begin{align*}
\widetilde{K}_{\lambda,\alpha}(z,w)=\frac{1}{\lambda+1}\sum_{n=0}^{\infty}
\frac{\Gamma(n+\lambda+\alpha+2)}{\Gamma(\alpha+1)\Gamma(n+\lambda+2)}
\phi_{n}(z)\overline{\phi_{n}(w)}.
\end{align*}
\end{lemma}

\begin{proof}
For $f(z)=\sum_{n=0}^{\infty}c_{n}\phi_{n}(z)$, it follows from (\ref{Tzphi-2}) that
$$
D_{z}\left(zf(z)\right)=\sum_{n=0}^{\infty}(n+\lambda+1)c_n\phi_{n}^{\lambda}(z),\qquad z\in\DD.
$$
Since $\{\phi_{n}(e^{i\theta})\}_{n=0}^{\infty}$ is an orthonormal set in $L_{\lambda}^2(\partial\DD)$, for $n\in\NN_0$ we have
\begin{align}\label{orthogonal-constant-1}
\int_{\DD}\left|\phi_{n}(z)\right|^2\,(1-|z|^2)^{\alpha}\,d\sigma_{\lambda}(z)
&=(2\lambda+2)\int_0^1r^{2n+2\lambda+1}(1-r^2)^{\alpha}\,dr \nonumber\\
&=\frac{(\lambda+1)\Gamma(\alpha+1)\Gamma(n+\lambda+1)}{\Gamma(n+\lambda+\alpha+2)},
\end{align}
so that
\begin{align*}
\int_{\DD}\overline{\phi_{n}(z)}D_{z}\left(zf(z)\right)\,(1-|z|^2)^{\alpha}\,d\sigma_{\lambda}(z)
=\frac{(\lambda+1)\Gamma(\alpha+1)\Gamma(n+\lambda+2)}{\Gamma(n+\lambda+\alpha+2)}\,c_n
\end{align*}
for $(1-|z|^{2})^{\alpha} D_{z}\left(zf(z)\right)\in L^{1}_{\lambda}(\DD)$. Finally termwise integration for
$D_{w}\left(wf(w)\right)\widetilde{K}_{\lambda,\alpha}(z,w)$ over $\DD$ with respect to the measure $(1-|w|^2)^{\alpha}d\sigma_{\lambda}(w)$ proves
\begin{align*}
\int_{\DD}D_{w}\left(wf(w)\right)\widetilde{K}_{\lambda,\alpha}(z,w)(1-|w|^{2})^{\alpha}\,d\sigma_{\lambda}(w)
=\sum_{n=0}^{\infty}c_{n}\phi_{n}(z)=f(z),\qquad z\in\DD.
\end{align*}
The proof of the lemma is finished.
\end{proof}

\begin{theorem}\label{derivative-Bloch-a}
If $f$ is $\lambda$-analytic in $\DD$ and $n\in\NN$ but $n\ge2$, then $f\in{\mathfrak{B}}_{\lambda}(\DD)$ if and only if $(1-|z|^{2})^n(D_{z}\circ z)^nf(z)$ is bounded on $\DD$; and moreover
\begin{align}\label{derivative-Bloch-1}
\|f\|_{{\mathfrak{B}}}\asymp\sup_{z\in\DD}(1-|z|^{2})^n|(D_{z}\circ z)^nf(z)|.
\end{align}
\end{theorem}

\begin{proof}
Suppose $f\in{\mathfrak{B}}_{\lambda}(\DD)$. By the formula (\ref{representation-1}),
\begin{align*}
(D_{z}\circ z)^nf(z)=\int_{\DD}D_{w}\left(wf(w)\right)(1-|w|^{2})\left[(D_{z}\circ z)^n\widetilde{K}_{\lambda}(z,w)\right]d\sigma_{\lambda}(w),
\end{align*}
so that
\begin{align*}
|(D_{z}\circ z)^nf(z)|\le \|f\|_{{\mathfrak{B}}}
\int_{\DD}\left|(D_{z}\circ z)^n\widetilde{K}_{\lambda}(z,w)\right|d\sigma_{\lambda}(w),\qquad z\in\DD.
\end{align*}
But from (\ref{Tzphi-2}) and (\ref{kernel-1}) it follows that
\begin{align*}
(D_{z}\circ z)^n\widetilde{K}_{\lambda}(z,w)=\frac{1}{\lambda+1}\sum_{k=0}^{\infty}
(k+\lambda+2)(k+\lambda+1)^n\phi_{k}(z)\overline{\phi_{k}(w)},
\end{align*}
and by Lemma \ref{h-kernal-a}(i) with $\beta=n+1$,
\begin{align*}
\left|(D_{z}\circ z)^n\widetilde{K}_{\lambda}(z,w)\right|\lesssim \frac{|1-z\overline{w}|^{-1}}{(|1-z\overline{w}|+|1-zw|)^{2\lambda}}
\left(\frac{1}{|1-z\overline{w}|^{n+1}}+\frac{1}{|1-zw|^{n+1}}\right)
\end{align*}
for $z,w\in\DD$. Thus for $z\in\DD$,
\begin{align*}
|(D_{z}\circ z)^nf(z)|
\lesssim
\int_{\DD}\frac{\|f\|_{{\mathfrak{B}}}|1-z\overline{w}|^{-1}}{(|1-z\overline{w}|+|1-zw|)^{2\lambda}}
\left(\frac{1}{|1-z\overline{w}|^{n+1}}+\frac{1}{|1-zw|^{n+1}}\right)d\sigma_{\lambda}(w).
\end{align*}
For $z=re^{i\theta}$, $w=se^{i\varphi}\in\DD$, on account of the inequalities in (\ref{inequality-1}) and (\ref{inequality-2}) we have
\begin{align*}
|(D_{z}\circ z)^nf(z)|
\lesssim \|f\|_{{\mathfrak{B}}}\int_0^1\int_{-\pi}^{\pi}
\left[\Psi_{r,\theta}(s,\varphi)+\Psi_{r,\theta}(s,-\varphi)\right]|\sin\varphi|^{2\lambda}d\varphi ds,
\end{align*}
where
\begin{align*}
 \Psi_{r,\theta}(s,\varphi)
 =\frac{\left(1-rs+\left|\sin(\theta-\varphi)/2\right|\right)^{-n-2}}{\left(1-rs+|\sin\theta|+|\sin\varphi|\right)^{2\lambda}}.
\end{align*}
Thus
\begin{align*}
|(D_{z}\circ z)^nf(z)|
\lesssim \|f\|_{{\mathfrak{B}}}\int_0^1\int_{-\pi}^{\pi}\frac{d\varphi ds}{\left(1-rs+\left|\sin(\theta-\varphi)/2\right|\right)^{n+2}},
\end{align*}
and after elementary calculations,
\begin{align*}
|(D_{z}\circ z)^nf(z)|
\lesssim \frac{\|f\|_{{\mathfrak{B}}}}{(1-r)^n},
\end{align*}
so that $(1-|z|^2)^n|(D_{z}\circ z)^nf(z)|\lesssim\|f\|_{{\mathfrak{B}}}$ for $|z|<1$.

Conversely, assume that $(1-|z|^{2})^n(D_{z}\circ z)^nf(z)$ is bounded on $\DD$. We shall prove, for $\ell=2,\cdots,n$,
\begin{align}\label{derivative-Bloch-2}
\sup_{z\in\DD}(1-|z|^2)^{\ell-1}|(D_{z}\circ z)^{\ell-1}f(z)|
\lesssim \sup_{z\in\DD}(1-|z|^2)^{\ell}|(D_{z}\circ z)^{\ell}f(z)|,
\end{align}
so that
$$
\sup_{z\in\DD}(1-|z|^2)|(D_{z}\circ z)f(z)|
\lesssim \sup_{z\in\DD}(1-|z|^2)^n|(D_{z}\circ z)^nf(z)|
$$
by descending induction. Thus (\ref{derivative-Bloch-1}) is gained.

To show (\ref{derivative-Bloch-2}), we apply (\ref{representation-2}) to the function $(D_{z}\circ z)^{\ell-1}f(z)$ instead of $f$ and with $\alpha=\ell$, to obtain
\begin{align*}
(D_{z}\circ z)^{\ell-1}f(z)=\int_{\DD}(1-|w|^{2})^{\ell}(D_{w}\circ w)^{\ell}f(w)\widetilde{K}_{\lambda,\ell}(z,w)d\sigma_{\lambda}(w),\qquad z\in\DD,
\end{align*}
where
\begin{align*}
\widetilde{K}_{\lambda,\ell}(z,w)=\frac{1}{\lambda+1}\sum_{k=0}^{\infty}
\frac{\Gamma(k+\lambda+\ell+2)}{\Gamma(\ell+1)\Gamma(k+\lambda+2)}
\phi_{k}(z)\overline{\phi_{k}(w)}.
\end{align*}
Thus for $z\in\DD$,
\begin{align}\label{derivative-Bloch-3}
|(D_{z}\circ z)^{\ell-1}f(z)|\le \sup_{\zeta\in\DD}(1-|\zeta|^2)^{\ell}|(D_{\zeta}\circ \zeta)^{\ell}f(\zeta)| \int_{\DD}|\widetilde{K}_{\lambda,\ell}(z,w)|d\sigma_{\lambda}(w).
\end{align}
By Lemma \ref{h-kernal-a}(i) with $\beta=\ell$,
\begin{align*}
\int_{\DD}|\widetilde{K}_{\lambda,\ell}(z,w)|d\sigma_{\lambda}(w)
\lesssim
\int_{\DD}\frac{|1-z\overline{w}|^{-1}}{(|1-z\overline{w}|+|1-zw|)^{2\lambda}}
\left(\frac{1}{|1-z\overline{w}|^{\ell}}+\frac{1}{|1-zw|^{\ell}}\right)d\sigma_{\lambda}(w).
\end{align*}
For $z=re^{i\theta}$, $w=se^{i\varphi}\in\DD$, on account of the inequalities in (\ref{inequality-1}) and (\ref{inequality-2}) we have
\begin{align*}
\int_{\DD}|\widetilde{K}_{\lambda,\ell}(z,w)|d\sigma_{\lambda}(w)
\lesssim\int_0^1\int_{-\pi}^{\pi}
\left[\widetilde{\Psi}_{r,\theta}(s,\varphi)+\widetilde{\Psi}_{r,\theta}(s,-\varphi)\right]|\sin\varphi|^{2\lambda}d\varphi ds,
\end{align*}
where
\begin{align*}
 \widetilde{\Psi}_{r,\theta}(s,\varphi)
 =\frac{\left(1-rs+\left|\sin(\theta-\varphi)/2\right|\right)^{-\ell-1}}{\left(1-rs+|\sin\theta|+|\sin\varphi|\right)^{2\lambda}}.
\end{align*}
Thus
\begin{align*}
\int_{\DD}|\widetilde{K}_{\lambda,\ell}(z,w)|\,d\sigma_{\lambda}(w)
\lesssim\int_0^1\int_{-\pi}^{\pi}\frac{d\varphi ds}{\left(1-rs+\left|\sin(\theta-\varphi)/2\right|\right)^{\ell+1}},
\end{align*}
and again, direct calculations show
\begin{align*}
\int_{\DD}|\widetilde{K}_{\lambda,\ell}(z,w)|\,d\sigma_{\lambda}(w)
\lesssim \frac{1}{(1-r)^{\ell-1}}, \qquad z=re^{i\theta}\in\DD.
\end{align*}
Finally applying this to (\ref{derivative-Bloch-3}) gives
$$
(1-|z|^2)^{\ell-1}|(D_{z}\circ z)^{\ell-1}f(z)|\lesssim\sup_{\zeta\in\DD}(1-|\zeta|^2)^{\ell}|(D_{\zeta}\circ \zeta)^{\ell}f(\zeta)|
$$
for $|z|<1$. The inequality (\ref{derivative-Bloch-2}) is proved, and the proof of the theorem is finished.
\end{proof}

\section{Boundedness of an integral operator from $L^{\infty}(\DD)$ onto ${\mathfrak{B}}_{\lambda}$}

We consider a general integral operator $T_{\lambda,\alpha}$ involving the parameter $\alpha$, of which the $\lambda$-Bergman projection $P_{\lambda}$ (cf. (\ref{Bergman-k-1}) and (\ref{Bergman-projection-1})) is a special case, i.\,e., $T_{\lambda,0}=P_{\lambda}$. It will be proved that $T_{\lambda,\alpha}$ for $\alpha>-1$ is bounded from $L^{\infty}(\DD)$ onto the $\lambda$-Bloch space ${\mathfrak{B}}_{\lambda}(\DD)$, and as an application, the dual of the $\lambda$-Bergman space $A^1_{\lambda}(\DD)$ is isomorphic to ${\mathfrak{B}}_{\lambda}(\DD)$.

For $\alpha>-1$, we consider the operator $T_{\lambda,\alpha}$ defined by
\begin{align}\label{integral-operator-1}
(T_{\lambda,\alpha}f)(z)=
\int_{\DD}f(w)K_{\lambda,\alpha}(z,w)(1-|w|^2)^{\alpha}d\sigma_{\lambda}(w),\qquad z\in\DD,
\end{align}
where
\begin{align}\label{integral-kernel-1}
K_{\lambda,\alpha}(z,w)=\frac{1}{\lambda+1}\sum_{n=0}^{\infty}
\frac{\Gamma(n+\lambda+\alpha+2)}{\Gamma(\alpha+1)\Gamma(n+\lambda+1)}
\phi_{n}(z)\overline{\phi_{n}(w)}.
\end{align}

The following lemma is necessary.

\begin{lemma}\label{integral-operator-a}
If $f\in{\mathfrak{B}}_{\lambda}(\DD)$, then the function $\psi$ given by
\begin{align}\label{psi-1}
\psi(z)=\frac{1-|z|^2}{\alpha+1}\left[D_{z}\left(zf(z)\right)+(\alpha+1)f(z)\right]
\end{align}
is bounded on $\DD$ and $\|\psi\|_{L^{\infty}}\lesssim\|f\|_{{\mathfrak{B}}_{\lambda}}$. Moreover $T_{\lambda,\alpha}\psi=f$.
\end{lemma}

\begin{proof}
By Definition \ref{Bloch-definition-a} and Proposition \ref{Bloch-b}, one has $\|\psi\|_{L^{\infty}}\lesssim\|f\|_{{\mathfrak{B}}_{\lambda}}$. Thus it remains to show $P_{\lambda,\alpha}\psi=f$ on $\DD$.

Assume that $f(z)=\sum_{n=0}^{\infty}c_{n}\phi_{n}(z)$. It follows from (\ref{Tzphi-2}) and (\ref{psi-1}) that
$$
\psi(z)=(1-|z|^2)\sum_{n=0}^{\infty}\frac{n+\lambda+\alpha+2}{\alpha+1}c_n\phi_{n}^{\lambda}(z), \qquad z\in\DD.
$$
By means of orthogonality of $\{\phi_{n}(e^{i\theta})\}_{n=0}^{\infty}$ in $L_{\lambda}^2(\partial\DD)$, for $n\in\NN_0$ we have
\begin{align*}
\int_{\DD}\psi(z)\overline{\phi_{n}(z)}(1-|z|^2)^{\alpha}\,d\sigma_{\lambda}(z)
=\frac{n+\lambda+\alpha+2}{\alpha+1}\,c_n\int_{\DD}\left|\phi_{n}(z)\right|^2\,(1-|z|^2)^{\alpha+1}d\sigma_{\lambda}(z),
\end{align*}
and then, on account of (\ref{orthogonal-constant-1}) with $\alpha+1$ instead of $\alpha$,
\begin{align*}
\int_{\DD}\psi(z)\overline{\phi_{n}(z)}(1-|z|^2)^{\alpha}\,d\sigma_{\lambda}(z)
=\frac{(\lambda+1)\Gamma(\alpha+1)\Gamma(n+\lambda+1)}{\Gamma(n+\lambda+\alpha+2)}\,c_n.
\end{align*}

Now from (\ref{integral-operator-1}) and (\ref{integral-kernel-1}), termwise integration for
$\psi(w)K_{\lambda,\alpha}(z,w)(1-|w|^2)^{\alpha}$ over $\DD$ with respect to the measure $d\sigma_{\lambda}(w)$ gives
\begin{align*}
(T_{\lambda,\alpha}\psi)(z)=
\int_{\DD}\psi(w)K_{\lambda,\alpha}(z,w)(1-|w|^2)^{\alpha}\,d\sigma_{\lambda}(w)
=\sum_{n=0}^{\infty}c_{n}\phi_{n}(z)=f(z)
\end{align*}
for $z\in\DD$. The proof of the lemma is completed.
\end{proof}

%The main result of this section is the following theorem.

\begin{theorem}\label{integral-operator-b}
For $\alpha>-1$, the operator $T_{\lambda,\alpha}$ defined by (\ref{integral-operator-1}) and (\ref{integral-kernel-1}) is bounded from $L^{\infty}(\DD)$ onto the $\lambda$-Bloch space ${\mathfrak{B}}_{\lambda}(\DD)$.
\end{theorem}

\begin{proof}
According to Lemma \ref{integral-operator-a}, it suffices to prove the boundedness of the operator $T_{\lambda,\alpha}$ from $L^{\infty}(\DD)$ into ${\mathfrak{B}}_{\lambda}(\DD)$.

For $\psi\in L^{\infty}(\DD)$, set $f(z)=(T_{\lambda,\alpha}\psi)(z)$, i.\,e.,
\begin{align*}
f(z)=
\int_{\DD}\psi(w)K_{\lambda,\alpha}(z,w)(1-|w|^2)^{\alpha}d\sigma_{\lambda}(w),\qquad z\in\DD.
\end{align*}
It follows that
\begin{align}\label{integral-operator-2}
D_{z}\left(zf(z)\right)=\int_{\DD}\psi(w)D_{z}\left(zK_{\lambda,\alpha}(z,w)\right)(1-|w|^2)^{\alpha}d\sigma_{\lambda}(w),\qquad z\in\DD;
\end{align}
and from (\ref{Tzphi-2}) and (\ref{integral-kernel-1}),
\begin{align*}
D_{z}\left(zK_{\lambda,\alpha}(z,w)\right)=\sum_{n=0}^{\infty}
\frac{(n+\lambda+1)\Gamma(n+\lambda+\alpha+2)}{(\lambda+1)\Gamma(\alpha+1)\Gamma(n+\lambda+1)}
\phi_{n}(z)\overline{\phi_{n}(w)}, \qquad |zw|<1.
\end{align*}

Since
\begin{align*}
\frac{(n+\lambda+1)\Gamma(n+\lambda+\alpha+2)}{(\lambda+1)\Gamma(\alpha+1)\Gamma(n+\lambda+1)}
=\sum_{j=0}^{M}c_j(n+1)^{\alpha+2-j}+O\left((n+1)^{\alpha+1-M}\right),
\end{align*}
where $M=[\alpha+2\lambda+3]$, appealing to Lemma \ref{h-kernal-a} (i) with $\beta=\alpha+2$ we have
\begin{align*}
\left|D_{z}\left(zK_{\lambda,\alpha}(z,w)\right)\right|\lesssim \frac{(|1-z\overline{w}|+|1-zw|)^{-2\lambda}}{|1-z\overline{w}|}
\left(\frac{1}{|1-z\overline{w}|^{\alpha+2}}+\frac{1}{|1-zw|^{\alpha+2}}\right)
\end{align*}
for $|zw|<1$.
Furthermore, with $z=re^{i\theta}$, $w=se^{i\varphi}\in\DD$, using the inequalities (\ref{inequality-1}) and (\ref{inequality-2}) we get
\begin{align*}
\left|D_{z}\left(zK_{\lambda,\alpha}(z,w)\right)\right|\lesssim
\widetilde{\Phi}_{r,\theta}(s,\varphi)+\widetilde{\Phi}_{r,\theta}(s,-\varphi),
\end{align*}
where
\begin{align*}
 \widetilde{\Phi}_{r,\theta}(s,\varphi)
 =\frac{\left(1-rs+\left|\sin(\theta-\varphi)/2\right|\right)^{-\alpha-3}}{\left(1-rs+|\sin\theta|+|\sin\varphi|\right)^{2\lambda}}.
\end{align*}
Now substituting this into (\ref{integral-operator-2}) yields, for $z=re^{i\theta}\in\DD$,
\begin{align*}
\left|D_{z}\left(zf(z)\right)\right|
\lesssim \|\psi\|_{L^{\infty}}\int_0^1\int_{-\pi}^{\pi}\frac{(1-s)^{\alpha}}{(1-rs+|\sin(\theta-\varphi)/2|)^{\alpha+3}}\,d\varphi ds.
\end{align*}

Note that the critical case to be considered is that for $-1<\alpha<0$. We take integration by parts with respect to $s$, to obtain
\begin{align*}
\left|D_{z}\left(zf(z)\right)\right|
\lesssim \|\psi\|_{L^{\infty}}\int_{-\pi}^{\pi}\left(1+\int_0^1\frac{(1-s)^{\alpha+1}}{(1-rs+|\sin(\theta-\varphi)/2|)^{\alpha+4}}\,ds\right)d\varphi.
\end{align*}
Consequently,
\begin{align*}
\left|D_{z}\left(zf(z)\right)\right|
\lesssim \|\psi\|_{L^{\infty}}\left(1+\int_{-\pi}^{\pi}\int_0^1\frac{dsd\varphi}{(1-rs+|\sin(\theta-\varphi)/2|)^{3}}\right).
\end{align*}
This again implies
\begin{align*}
\left|D_{z}\left(zf(z)\right)\right|
\lesssim \frac{\|\psi\|_{L^{\infty}}}{1-r}, \qquad z\in\DD,
\end{align*}
so that $(1-|z|^2)|D_{z}\left(zf(z)\right)|\lesssim\|\psi\|_{L^{\infty}}$ for $z\in\DD$. Therefore $T_{\lambda,\alpha}\psi=f\in{\mathfrak{B}}_{\lambda}(\DD)$ and $\|T_{\lambda,\alpha}\psi\|_{{\mathfrak{B}}_{\lambda}}\lesssim\|\psi\|_{L^{\infty}}$.
The proof of the theorem is completed.
\end{proof}

%Note that the operator $T_{\lambda,\alpha}$ is an analog of the operator considered in \cite{Ch1} on the usual analytic function spaces.

We have the following corollary immediately.

\begin{corollary}\label{Bergman-projection-Bloch-a}
The $\lambda$-Bergman projection $P_{\lambda}$, defined by (\ref{Bergman-projection-1}), is a bounded operator from $L^{\infty}(\DD)$ onto the $\lambda$-Bloch space ${\mathfrak{B}}_{\lambda}(\DD)$.
\end{corollary}

Finally we turn to the dual relation of the $\lambda$-Bergman space $A^1_{\lambda}(\DD)$ and ${\mathfrak{B}}_{\lambda}(\DD)$.

\begin{theorem}\label{duality-Bergman-Bloch-a}
The dual space $A^1_{\lambda}(\DD)^*$ is isomorphic to the $\lambda$-Bloch space ${\mathfrak{B}}_{\lambda}(\DD)$ in the sense that, each $L\in A^1_{\lambda}(\DD)^*$ can be represented by
\begin{align*}
L(f)=\lim_{t\rightarrow1^-}\int_{t\DD}f(z)\overline{g(z)}\,d\sigma_{\lambda}(z),\qquad f\in A^1_{\lambda}(\DD),
\end{align*}
with a unique function $g\in{\mathfrak{B}}_{\lambda}(\DD)$ satisfying
$$
C'\|g\|_{{\mathfrak{B}}_{\lambda}}\le\|L\|\le C''\|g\|_{{\mathfrak{B}}_{\lambda}},
$$
where the constants $C'$ and $C''$ are independent of $g$.
\end{theorem}

\begin{proof}
Assume that $f$ and $g$ are $\lambda$-analytic in $\DD$, and for $s\in(0,1)$, set $f_s(z)=f(sz)$ and $g_s(z)=g(sz)$. Applying Lemma \ref{representation-a} to $g_s$ we have
\begin{align}\label{dual-representation-1}
\int_{\DD}f_s(z)\overline{g_s(z)}d\sigma_{\lambda}(z)=\int_{\DD}F_s(w)\overline{D_{w}\left(wg_s(w)\right)}(1-|w|^{2})d\sigma_{\lambda}(w),
\end{align}
where $F_s(z)=F(sz)$ and
\begin{align*}
F(z)=\int_{\DD}f(w)\widetilde{K}_{\lambda}(z,w)\,d\sigma_{\lambda}(w).
\end{align*}
If $f(z)=\sum_{n=0}^{\infty}c_{n}\phi_{n}(z)$, it follows from \cite[(20)]{LW1} that
\begin{align*}
c_n=\frac{n+\lambda+1}{\lambda+1}\int_{\DD}f(w)\overline{\phi_{n}(w)}\,d\sigma_{\lambda}(w),
\end{align*}
and then, from (\ref{kernel-1}),
\begin{align*}
F(z)=f(z)+\sum_{n=0}^{\infty}\frac{1}{n+\lambda+1}c_n\phi_{n}(z).
\end{align*}
But by \cite[Theorem 4.4]{QW1}, $\{(n+\lambda+1)^{-1}\}_{n=0}^{\infty}$ is a multiplier from $A^1_{\lambda}(\DD)$ to $H_{\lambda}^1(\DD)$, so that
$F\in A^{1}_{\lambda}(\DD)$ whenever $f\in A^{1}_{\lambda}(\DD)$, and $\|F\|_{A_{\lambda}^1}\lesssim\|f\|_{A_{\lambda}^1}$.

Now on the two sides of the equation (\ref{dual-representation-1}), we make substitution of variables as $z\mapsto z/s$ and $w\mapsto w/s$ respectively, to get
\begin{align*}
\int_{s\DD}f(z)\overline{g(z)}d\sigma_{\lambda}(z)=\frac{1}{s^2}\int_{s\DD}F(w)\overline{D_{w}\left(wg(w)\right)}(s^2-|w|^{2})d\sigma_{\lambda}(w).
\end{align*}
For $f\in A^{1}_{\lambda}(\DD)$ and $g\in {\mathfrak{B}}_{\lambda}(\DD)$,
$$
\left|F(w)\overline{D_{w}\left(wg(w)\right)}(s^2-|w|^{2})\right|\le\|g\|_{{\mathfrak{B}}_{\lambda}}|F(w)|\in L^{1}_{\lambda}(\DD),
$$
and then, by Lebesgue's dominated convergence theorem, the linear functional
\begin{align}\label{A1-functional-2}
L(f):=\lim_{s\rightarrow1^-}\int_{s\DD}f(z)\overline{g(z)}\,d\sigma_{\lambda}(z)=\int_{\DD}F(w)\overline{D_{w}\left(wg(w)\right)}(1-|w|^{2})d\sigma_{\lambda}(w)
\end{align}
is well defined for $f\in A^1_{\lambda}(\DD)$, and moreover $|L(f)|\lesssim\|g\|_{{\mathfrak{B}}_{\lambda}}\|f\|_{A_{\lambda}^1}$.

Conversely, by the Hahn-Banach theorem every $L\in A^1_{\lambda}(\DD)^*$ can be extended to a bounded linear functional on $L^1_{\lambda}(\DD)$ with the same norm, and the Riesz representation theorem implies that $L$ has the following representation
\begin{align*}
L(f)=\int_{\DD}f(z)h(z)d\sigma_{\lambda}(z)
\end{align*}
for all $f\in L^1_{\lambda}(\DD)$, with some $h\in L^{\infty}(\DD)$ satisfying $\|h\|_{L^{\infty}}=\|L\|$.

%Conversely, by the Hahn-Banach theorem every functional $L\in A^1_{\lambda}(\DD)^*$ can be extended to a bounded linear functional on $L^1_{\lambda}(\DD)$ with the same norm, and the Riesz representation theorem yields that
%$L^1_{\lambda}(\DD)^*$ is isometrically isomorphic to $L^{\infty}(\DD)$ in the sense that, each
%$L\in L^1_{\lambda}(\DD)^*$ has a representation as
%\begin{align}\label{L1-functional-2}
%L(f)=\int_{\DD}f(z)h(z)d\sigma_{\lambda}(z)
%\end{align}
%for all $f\in L^1_{\lambda}(\DD)$, with a unique function $h\in L^{\infty}(\DD)$ satisfying $\|h\|_{L^{\infty}}=\|L\|$.

If we put $g(z)=(P_{\lambda}\bar{h})(z)$, then by Corollary \ref{Bergman-projection-Bloch-a}, $g\in {\mathfrak{B}}_{\lambda}(\DD)$ and
\begin{align*}
\|g\|_{{\mathfrak{B}}_{\lambda}}\lesssim\|h\|_{L^{\infty}}=\|L\|.
\end{align*}
The function $g$, by what is just proved, defines a bounded linear functional on $A^1_{\lambda}(\DD)$ according to (\ref{A1-functional-2}); for the sake of distinction, such a functional is denoted by $\widetilde{L}$.

It remains to show that,
for $f\in A^{1}_{\lambda}(\DD)$,
\begin{align}\label{A1-functional-3}
L(f)=\widetilde{L}(f).
\end{align}
For the purpose, it suffices to take $f$ to be a $\lambda$-analytic polynomial by the density result in \cite[Theorem 5.3]{LW1}. But then we can write
\begin{align*}
\widetilde{L}(f)=\int_{\DD}f(z)\overline{g(z)}\,d\sigma_{\lambda}(z)
=\int_{\DD}\int_{\DD}f(z)h(w)K_{\lambda}(w,z)\,d\sigma_{\lambda}(w)d\sigma_{\lambda}(z);
\end{align*}
and furthermore, Fubini's theorem and the reproducing formula (\ref{reproducing-Bergman-1}) readily give us
\begin{align*}
\widetilde{L}(f)=\int_{\DD}f(w)h(w)\,d\sigma_{\lambda}(w)=L(f),
\end{align*}
that is the equality (\ref{A1-functional-3}). The uniqueness is a consequence of the norm equivalence, and the proof of the theorem is finished.
\end{proof}

 \vskip .2in

{\bf Author contributions:}  All authors have contributed equally on the manuscript.
 \vskip .1in

{\bf Data Availability Statement:}  No datasets were generated or analyzed during the current study.
 \vskip .2in

{\bf Declarations}
 \vskip .1in

{\bf Conflict of interest:} The authors declare no conflict of interest.

\end{document}